\documentclass[11pt]{article}
\usepackage{amsmath,mathrsfs}
\usepackage{amssymb}
\usepackage{amsthm}
\usepackage{mathrsfs}
\usepackage[all]{xy}
\usepackage{indentfirst}
\usepackage{color}
\allowdisplaybreaks

\newtheorem{theorem}{Theorem}[section]
\newtheorem*{theorem*}{Theorem}

\newtheorem{remark}[theorem]{Remark}
\newtheorem{lemma}[theorem]{Lemma}
\newtheorem{definition}[theorem]{Definition}

\newtheorem{proposition}[theorem]{Proposition}

\newtheorem*{proposition*}{Proposition}

\newcommand{\R}{\mathbb{R}}

\newcommand{\be}{\begin{eqnarray*}}
\newcommand{\ee}{\end{eqnarray*}}
\newcommand{\ba}{\begin{align*}}
\newcommand{\bpm}{\begin{pmatrix}}
\newcommand{\epm}{\end{pmatrix}}

\newcommand{\bx}{\boldsymbol{x}}

\begin{document}
\title{On the Fueter-Sce  theorem for generalized partial-slice monogenic functions}

\author{Zhenghua Xu$^1$ \thanks{This work was partially supported by  the Anhui Provincial Natural Science Foundation (No. 2308085MA04) and the National Natural Science Foundation of China (No. 11801125).} Irene Sabadini$^2$ \thanks{This work was partially supported by PRIN 2022 {\em Real and Complex Manifolds: Geometry and Holomorphic Dynamics}.}\\
\emph{$^1$
\small School of Mathematics, Hefei University of Technology,}\\
\emph{\small  Hefei, 230601, P.R. China}\\
\emph{\small E-mail address: zhxu@hfut.edu.cn}
\\
\emph{\small $^2$ Politecnico di Milano, Dipartimento di Matematica,}\\
\emph{\small Via E. Bonardi, 9, 20133 Milano Italy} \\
\emph{\small E-mail address:   irene.sabadini@polimi.it}
}

\maketitle

\begin{abstract}
In a recent paper, we introduced the concept of generalized partial-slice monogenic functions. The class of these functions includes both the theory of monogenic functions and of slice monogenic functions with values in a Clifford algebra. In this paper, we prove a version of the  Fueter-Sce theorem in this new setting, which allows to construct monogenic functions in higher dimensions starting from generalized partial-slice monogenic functions. We also prove that an alternative construction can be obtained by using the dual Radon transform. To this end, we provide a study of the generalized CK-extension.
\end{abstract}
\section{Introduction}
In quaternionic analysis, the celebrated Fueter  theorem asserts that every holomorphic  function of one complex variable induces  a regular function, see \cite{Fueter}. Fueter's result is based on a two steps procedure. The first step consists in considering  a holomorphic function $f_{0}(z)=u(x,y)+iv(x,y)$ $(u,v\in \mathbb{R},z=x+yi\in \mathbb{C})$  defined on a domain $D\subseteq \mathbb{C}$ symmetric with respect to the real axis and satisfying $\overline{f_{0}(z)}=f_{0}(\overline{z})$; then one constructs the induced quaternionic-valued function of the quaternionic variable:
$$\overrightarrow{f_{0}}(q)=u(q_{0},|\underline{q}|)+ \frac{\underline{q}}{|\underline{q}|}v(q_{0},|\underline{q}|),$$
where $q=q_{0}+q_{1}i+q_{2}j+q_{3}k$, i.e.  $(q_{0}, q_{1},q_{2},q_{3})\in \mathbb{R}^{4}$  is  the real coordinate  of the quaternion $q$ with respect to  the standard basis $\{1,\,i,\,j, \,k\}$.  Here the quaternion  $q=q_{0}+q_{1}i+q_{2}j+q_{3}k$ is split into its real part $q_{0}$ and imaginary part $\underline{q}=q_{1}i+q_{2}j+q_{3}k$, with $|\underline{q}|=\sqrt{\sum_{m=1}^{3}q_{m}^2}$. The second step consists in applying the four-dimensional Laplacian $\Delta_q=\sum_{m=0}^{3}\frac{\partial^{2}}{\partial q_{m}^{2}}$ to $\overrightarrow{f_{0}}(q)$, where $q=q_{0}+q_{1}i+q_{2}j+q_{3}k$ is identified with the $4$-tuple $(q_{0}, q_{1},q_{2},q_{3})\in \mathbb{R}^{4}$. The function $\breve f(q)=\Delta_q \overrightarrow{f_{0}}(q)$
satisfies
 $$D_{q}\breve f(q)=\breve f(q)D_{q}=0,$$
 namely $\breve f(q)$ is both left and right regular,
where  $D_{q}=\frac{\partial}{\partial q_{0}}+i\frac{\partial}{\partial q_{1}}+j\frac{\partial}{\partial q_{2}}+k\frac{\partial}{\partial q_{3}}$  is the  Cauchy-Riemann-Fueter operator.

In 1957, the Fueter  theorem  was generalized by Sce \cite{Sce} (see \cite{Colombo-Sabadini-Struppa-20} for the English translation) to functions defined on open sets in the Euclidean space $\mathbb{R}^{n+1}$ for odd $n\in \mathbb{N}$.  Sce's result asserts in particular that, given a holomorphic  function $f_{0}$ as above, the induced function
$$\overrightarrow{f_{0}}(x)=u(x_{0},|\underline{x}|)+ \frac{\underline{x}}{|\underline{x}|}v(x_{0},|\underline{x}|)$$
satisfies
 $$D_{x}\Delta_{x}^{\frac{n-1}{2}} \overrightarrow{f_{0}}(x)=(\Delta_{x}^{\frac{n-1}{2}} \overrightarrow{f_{0}}(x))D_{x}=0,$$
where $x= x_{0}+\underline{x}=\sum_{m=0}^{n}x_{m}e_{m}\in \mathbb{R}^{n+1}$ with  $\{e_1,e_2,\ldots, e_n\}$ being the generators  of the Clifford algebra  $\mathbb{R}_{n}$,  $D_{x}= \sum_{m=0}^{n}e_{m}\frac{\partial}{\partial x_{m}}$  is the generalized  Cauchy-Riemann operator, and $\Delta_{x}= \sum_{m=0}^{n}\frac{\partial^{2}}{\partial x_{m}^{2}}$ is the Laplacian in $\mathbb{R}^{n+1}$.

In summary, Sce's result generates  monogenic functions from holomorphic functions. However, there are dimensional limitations to avoid  getting fractional powers of the Laplacian. By using the Fourier multiplier definition of the fractional Laplacian,  Qian   generalized   the Fueter's and Sce's result to any dimension \cite{Qian-97}. Nowadays, this deep result is known as the Fueter-Sce-Qian theorem. It is one of the most fundamental results in hypercomplex analysis, which has been considered in various settings, see e.g.  \cite{Dong-Qian-21,Eelbode,Fei,Kou,Pena-06,Sommen-00}   and the survey \cite{Qian-15}.

It is interesting to note that the intermediate step of the Fueter construction, namely the first step, was not studied as a function theory until the theory of slice monogenic functions was considered.

In fact, the induced function
$$\overrightarrow{f_{0}}(x)=u(x_{0},|\underline{x}|)+ \frac{\underline{x}}{|\underline{x}|}v(x_{0},|\underline{x}|)$$
is such that the pair $(u,v)$ satifies the Cauchy-Riemann system. If one allows the functions $u,v$ to have values in the Clifford algebra $\mathbb R_n$, it is immediate to see that the function $\overrightarrow{f_{0}}(x)$ is slice monogenic. This observation gave rise to various results and to applications in operator theory: we refer the reader to \cite{Colombo-Sabadini-Sommen-10,Colombo-Sabadini-Struppa-20} for more information and to \cite{Colombo-Sabadini-Sommen-13, Dong-Qian-16} for the related problem of the inversion of the Fueter-Sce-Qian map.

Recently, the present authors have introduced the concept of generalized partial-slice monogenic functions which has been designed to unify the two theories of monogenic functions and of slice monogenic functions; see \cite{Xu-Sabadini} and a subsequent work on conformal invariance \cite{Ding-Xu}.  In this article, we address the problem of  proving a version of the  Fueter-Sce  theorem for generalized partial-slice monogenic functions. The result provides  a method of constructing monogenic functions in higher dimensions starting from  monogenic functions in lower dimensions.
This result contains the classical Fueter-Sce result for slice monogenic functions as a particular case. It is worthwhile to note that it is available another method to construct monogenic functions starting from slice monogenic functions, making use of the Radon transform and its dual, see \cite{Colombo-Sabadini-Soucek-15}, so it is of interest to ask if the construction works also in the present framework. The answer is positive and we address this aspect in the last part of the paper.

The plan of the article is as follows. Section 2 contains some preliminary notions on Clifford algebras and on generalized partial-slice monogenic functions. In Section 3,  we prove the counterpart of the Fueter-Sce  theorem for generalized partial-slice monogenic functions, see Theorems \ref{Fueter-theorem} and \ref{Fueter-theorem-GSM}.  In Section 4, following \cite{Xu-Sabadini}, we introduce a slice  Cauchy-Kovalevskaya extension of real analytic functions defined on some domain in $\mathbb R^{p+1}$ and then describe a Fueter-Sce theorem in connection with the generalized CK-extension in Theorem \ref{CK-Fueter-relation}. Finally, in Section 5 we prove a connection (Theorem \ref{GCK-Dual-relation}) between the  generalized CK-extension and the dual Radon transform for generalized partial-slice monogenic functions.

\section{Preliminaries}
In this section, we collect some preliminary results on Clifford algebras and on   generalized partial-slice monogenic
functions recently introduced and studied in \cite{Xu-Sabadini}, to which we refer the reader for more information.
\subsection{Clifford algebras}
Let $\{e_1,e_2,\ldots, e_n\}$ be a standard orthonormal basis for the $n$-dimensional real Euclidean space  $\mathbb{R}^n$. A real Clifford algebra, denoted by $\mathbb{R}_{n}$, is generated  by these basis elements assuming that   $$e_i e_j + e_j e_i= -2\delta_{ij}, \quad 1\leq i,j\leq n,$$
where $\delta_{ij}$ is the Kronecker symbol. \\
Hence,  every element in  $\mathbb{R}_{n}$ can be written as
 $$a=\sum_A a_Ae_A, \quad a_A\in \mathbb{R},$$
 where
$$e_A=e_{j_1}e_{j_2}\cdots e_{j_r},$$
and $A=\{j_1,j_2, \cdots, j_r\}\subseteq\{1, 2, \cdots, n\}$ and $1\leq j_1< j_2 < \cdots < j_r \leq n,$  $ e_\emptyset=e_0=1$.

As a real  vector space, the dimension of Clifford algebra $\mathbb{R}_n$  is $2^{n}$. The norm of $a$ is defined by $|a|= ({\sum_{A}|a_{A}|^{2}} )^{\frac{1}{2}}.$ For $k=0,1,\ldots,n$, the real linear subspace $\mathbb{R}_n^k$ of $\mathbb{R}_n$, called $k$-vectors, is  generated by the $\begin{pmatrix} n\\k\end{pmatrix}$ elements of the form
 $$e_A=e_{i_1}e_{i_2}\cdots e_{i_k},\quad 1\leq i_1<i_2<\cdots<i_k\leq n.$$
In particular, the element $a_{0}=a_{\emptyset}$ is called scalar part of $a$.

The Clifford conjugation of $a$ is  defined as
$$\overline{a} =\sum_Aa_A\overline{e_A},$$
 where $\overline{e_{j_1}\ldots e_{j_r}}=\overline{e_{j_r}}\ldots\overline{e_{j_1}},\ \overline{e_j}=-e_j,1\leq j\leq n,\ \overline{e_0}=e_0=1$.

 A typical  subset of Clifford numbers $\mathbb{R}_n$ is the set of the so-called paravectors   $\mathbb{R}_n^{0} \oplus \mathbb{R}_n^{1}$. This subset will  be identified with $\mathbb{R}^{n+1}$ via the map
$$(x_0,x_1,\ldots,x_n) \longmapsto   x=x_{0}+\underline{x}=\sum_{i=0}^{n}e_ix_i.$$
For a paravector $x\neq0$, its norm is given by $|x|=(x\overline x)^{1/2}$  and so its inverse is given by $x^{-1}=  \overline{x}|x|^{-2}.$

\subsection{Generalized partial-slice monogenic functions}

In the sequel, let $p$ and $q$ be a nonnegative and a positive integer, respectively. We consider functions $f:\Omega\longrightarrow  \mathbb{R}_{p+q}$ where $\Omega\subseteq\R^{p+q+1}$ is a domain.

An element $\bx_p$ will be identified with a paravector in $\mathbb{R}_{p}$, while an element  $\bx\in\R^{p+q+1}$ will be identified with a paravector in $\mathbb{R}_{p+q}$, and we shall write it as
$$\bx=\bx_p+\underline{\bx}_q \in\R^{p+1}\oplus\R^q, \quad \bx_p=\sum_{i=0}^{p}x_i e_i,\ \underline{\bx}_q=\sum_{i=p+1}^{p+q}x_i e_i.$$
Throughout the paper, this splitting of $\bx\in\R^{p+q+1}$ as $\R^{p+1}\oplus\R^q$ is fixed.

Similarly, the generalized Cauchy-Riemann operator $D_{\bx}$ is split as

\begin{equation}\label{Dxx}
D_{\bx}=D_{\bx_p} +D_{\underline{\bx}_q}, \quad D_{\bx_p} =\sum_{i=0}^{p}e_i\partial_{x_i}, \
D_{\underline{\bx}_q} =\sum_{i=p+1}^{p+q}e_i\partial_{x_i}.
\end{equation}

Denote by $\mathbb{S}$ the sphere of unit $1$-vectors in $\mathbb R^q$, whose elements $(x_{p+1},\ldots, x_{p+q})$ are identified with $\underline{\bx}_q=\sum_{i=p+1}^{p+q}x_i e_i$, i.e.
$$\mathbb{S}=\big\{\underline{\bx}_q: \underline{\bx}_q^2 =-1\big\}=\big\{\underline{\bx}_q=\sum_{i=p+1}^{p+q}x_i e_i:\sum_{i=p+1}^{p+q}x_i^{2}=1\big\}.$$
Note that, for $\underline{\bx}_q\neq0$, there exists a uniquely determined $r\in \mathbb{R}^{+}=\{x\in \mathbb{R}: x>0\}$ and $\underline{\omega}\in \mathbb{S}$, such that $\underline{\bx}_q=r\underline{\omega}$, more precisely
 $$r=|\underline{\bx}_q|, \quad \underline{\omega}=\frac{\underline{\bx}_q}{|\underline{\bx}_q|}. $$
 When $\underline{\bx}_q= 0$, set $r=0$ and $\underline{\omega}$ is not uniquely defined, in fact for every $\underline{\omega}\in \mathbb{S}$ we have $\bx=\bx_p+\underline{\omega} \cdot 0$.

The upper half-space $\mathrm{H}_{\underline{\omega}}$ in $\mathbb{R}^{p+2}$ associated with $\underline{\omega}\in \mathbb{S}$ is defined by
$$\mathrm{H}_{\underline{\omega}}=\{\bx_p+r\underline{\omega}, \bx_p \in\R^{p+1}, r\geq0 \},$$
and it is clear from the previous discussion that
$$ \R^{p+q+1}=\bigcup_{\underline{\omega}\in \mathbb{S}} \mathrm{H}_{\underline{\omega}},$$
and
$$ \R^{p+1}=\bigcap_{\underline{\omega}\in \mathbb{S}} \mathrm{H}_{\underline{\omega}}.$$

In the sequel, we shall make use of the notation
$$
\Omega_{\underline{\omega}}:=\Omega\cap (\mathbb{R}^{p+1} \oplus \underline{\omega} \mathbb{R})\subseteq \mathbb{R}^{p+2}
$$
where $\Omega$ is a domain  in $\mathbb{R}^{p+q+1}$.

Now we introduce  the  conception of  generalized partial-slice monogenic functions \cite{Xu-Sabadini}.
\begin{definition}[Generalized partial-slice monogenic functions] \label{definition-slice-monogenic}
 Let $\Omega$ be a domain in $\mathbb{R}^{p+q+1}$. A function $f :\Omega \rightarrow \mathbb{R}_{p+q}$ is called left  generalized partial-slice monogenic of type $(p,q)$ if, for all $ \underline{\omega} \in \mathbb S$, its restriction $f_{\underline{\omega}}$ to $\Omega_{\underline{\omega}}\subseteq \mathbb{R}^{p+2}$  has continuous partial derivatives and  satisfies
$$D_{\underline{\omega}}f_{\underline{\omega}}(\bx):=(D_{\bx_p}+\underline{\omega}\partial_{r}) f_{\underline{\omega}}(\bx_p+r\underline{\omega})=0,$$
for all $\bx=\bx_p+r\underline{\omega} \in \Omega_{\underline{\omega}}$.
 \end{definition}
We denote by $\mathcal {GSM}(\Omega)$ (or $\mathcal {GSM}^{L}(\Omega)$ when needed) the function class of  all  left generalized partial-slice monogenic functions  of type $(p,q)$  in $\Omega$.  Throughout this article,  we always deal with left generalized partial-slice monogenic functions  of type $(p,q)$ and  hence omit to specify type $(p,q)$ when the context is clear.  It is immediate to verify that $\mathcal {GSM}^{L}(\Omega)$ is a right Clifford-module over $\mathbb R_{p+q}$.

Likewise, denote by $\mathcal {GSM}^{R}(\Omega)$ the left Clifford module of  all right  generalized partial-slice monogenic functions of type $(p,q)$ $f:\Omega  \rightarrow \mathbb{R}_{p+q}$ which are defined by requiring that the restriction $f_{\underline{\omega}}$ to $\Omega_{\underline{\omega}}$ satisfies
$$f_{\underline{\omega}}(\bx)D_{\underline{\omega}}:={f_{\underline{\omega}} (\bx_p+r\underline{\omega})D_{\bx_p}}+ \partial_{r}f_{\underline{\omega}} (\bx_p+r\underline{\omega})\underline{\omega}=0, \quad \bx=\bx_p+r\underline{\omega} \in \Omega_{\underline{\omega}},$$
for all  $ \underline{\omega} \in \mathbb S$.

Let $k\in \mathbb{N}\cup\{\infty\}.$ In a classical way, we denote by $C^k(\Omega)$ the space of $k$-times continuously differentiable $\mathbb{R}_{p+q}$-valued functions in $\Omega$. In particular,  $C^0(\Omega)$ ($C(\Omega)$ for short) is the space of  continuously  $\mathbb{R}_{p+q}$-valued functions in $\Omega$.  In the sequel, we will use the notion of monogenic functions, see \cite{Brackx}, which is recalled in the following  definition.
\begin{definition}[Monogenic functions]\label{monogenic-Clifford}
Let $\Omega $ be a domain in $\mathbb{R}^{p+q+1}$ and let  $f \in C^1(\Omega)$.    The function $f=\sum_{A}  e_{A}   f_{A}$ is called  (left)  monogenic in $\Omega $ if it satisfies the generalized Cauchy-Riemann equation
$$ D_{\bx}f(\bx)=\sum _{i=0}^{p+q}e_{i} \frac{\partial f}{\partial x_{i}}(\bx)
= \sum _{i=0}^{n} \sum_{A} e_{i}e_{A} \frac{\partial f_{A}}{\partial x_{i}}(\bx)=0, \quad \bx\in \Omega. $$
The set of monogenic functions on $\Omega$ will be denoted by $\mathcal{M}(\Omega)$.
\end{definition}

\begin{remark}\label{rem32}
{\rm
 When $(p,q)=(n-1,1)$, the  notion of generalized partial-slice monogenic functions in Definition  \ref{definition-slice-monogenic} coincides with the notion of  classical monogenic functions  defined in $\Omega\subseteq\mathbb{R}^{n+1}$   with values in the Clifford algebra $\mathbb{R}_{n}$, which is denoted by $\mathcal {M}(\Omega)$. For more details on the theory of monogenic functions, see for instance \cite{Brackx,Colombo-Sabadini-Sommen-Struppa-04,Delanghe-Sommen-Soucek,Gurlebeck}.

When $(p,q)=(0,n)$, Definition  \ref{definition-slice-monogenic}  boils down to that one of  slice monogenic functions defined in $\mathbb{R}^{n+1}$ and   with values in the Clifford algebra $\mathbb{R}_{n}$, which is denoted by $\mathcal{SM}(\Omega)$; see \cite{Colombo-Sabadini-Struppa-09} or \cite{Colombo-Sabadini-Struppa-11}.}
\end{remark}

\begin{definition} \label{slice-domain}
 Let $\Omega$ be a domain in $\mathbb{R}^{p+q+1}$.

1.   $\Omega$ is called  slice domain if $\Omega\cap\mathbb R^{p+1}\neq\emptyset$  and $\Omega_{\underline{\omega}}$ is a domain in $\mathbb{R}^{p+2}$ for every  $\underline{\omega}\in \mathbb{S}$.

2.   $\Omega$   is called  partially  symmetric with respect to $\mathbb R^{p+1}$ (p-symmetric for short) if, for   $\bx_{p}\in\R^{p+1}, r \in \mathbb R^{+},$ and $ \underline{\omega}  \in \mathbb S$,
$$\bx=\bx_p+r\underline{\omega} \in \Omega\Longrightarrow [\bx]:=\bx_p+r \mathbb S=\{\bx_p+r \underline{\omega}, \ \  \underline{\omega}\in \mathbb S\} \subseteq \Omega. $$
 \end{definition}

Denote by $\mathcal{Z}_{f}(\Omega)$  the zero set of the function $f:\Omega\subseteq\mathbb{R}^{n+1}\rightarrow \mathbb{R}_{n}$.

We recall   an identity theorem for generalized partial-slice monogenic functions over slice domains.
\begin{theorem}  {\bf(Identity theorem)}\label{Identity-theorem}
Let $\Omega\subseteq \mathbb{R}^{p+q+1}$ be a  slice domain and $f,g:\Omega\rightarrow \mathbb{R}_{p+q}$ be   generalized partial-slice monogenic functions.
If there is an imaginary $\underline{\omega}  \in \mathbb S $ such that $f=g$ on a   $(p+1)$-dimensional smooth manifold in $\Omega_{\underline{\omega}}$, then $f\equiv g$ in  $\Omega$.
\end{theorem}

The identity theorem for generalized partial-slice monogenic  functions allows to formulate  a representation formula.
\begin{theorem}  {\bf(Representation Formula)}  \label{Representation-Formula-SM}
Let $\Omega\subseteq \mathbb{R}^{p+q+1}$ be a p-symmetric slice domain and $f:\Omega\rightarrow \mathbb{R}_{p+q}$ be a  generalized partial-slice monogenic function.  Then, for any $\underline{\omega}\in \mathbb{S}$ and for $\bx_p+r\underline{\omega} \in \Omega$,
\begin{equation}\label{Representation-Formula-eq}
f(\bx_p+r \underline{\omega})=\frac{1}{2} (f(\bx_p+r\underline{\eta} )+f(\bx_p-r\underline{\eta}) )+
\frac{ 1}{2} \underline{\omega}\underline{\eta} (  f(\bx_p-r\underline{\eta} )-f(\bx_p+r\underline{\eta})),
\end{equation}
for any $\underline{\eta}\in \mathbb{S}$.

Moreover, the following two functions do not depend on $\underline{\eta}$:
$$F_1(\bx_p,r)=\frac{1}{2} (f(\bx_p+r\underline{\eta} )+f(\bx_p-r\underline{\eta} ) ),$$
$$F_2(\bx_p,r)=\frac{ 1}{2}\underline{\eta}(  f(\bx_p-r\underline{\eta} )-f(\bx_p+r\underline{\eta})).$$
\end{theorem}

\subsection{Generalized partial-slice  functions}
Throughout this article, let $D \subseteq \mathbb{R}^{p+2}$  be a domain, which is invariant under the reflection  of the $(p+2)$-th variable, i.e.
$$ \bx':=(\bx_p,r) \in D \Longrightarrow   \bx_\diamond':=(\bx_p,-r)  \in D.$$
The  \textit{p-symmetric completion} $ \Omega_{D}\subseteq\mathbb{R}^{p+q+1}$ of  $D$ is defined by
$$\Omega_{D}=\bigcup_{\underline{\omega} \in \mathbb{S}} \, \big \{\bx_p+r\underline{\omega}\  : \ \exists \ \bx_p \in \mathbb{R}^{p+1},\ \exists \ r\geq 0,\  \mathrm{s.t.} \ (\bx_p,r)\in D \big\}.$$

By Definition  \ref{slice-domain}, it is easy to see that a domain $\Omega \subseteq \mathbb{R}^{p+q+1}$ is p-symmetric if and only if $\Omega=\Omega_{D}$ for some domain $D \subseteq\mathbb{R}^{p+2}$.
\begin{definition}
Let $D\subseteq \mathbb{R}^{p+2}$ be  a domain, invariant under the reflection  of the $(p+2)$-th variable.
A function $f: \Omega_D\longrightarrow  \mathbb{R}_{p+q}$ of the form
  \begin{equation}\label{genpslice}
  f(\bx)=F_1(\bx')+\underline{\omega} F_2(\bx'), \qquad   \bx=\bx_p+r\underline{\omega}   \in \Omega_{D},\ \underline{\omega}\in \mathbb{S},
  \end{equation}
 where the $\mathbb{R}_{p+q} $-valued components  $F_1, F_2$ of $f$satisfy
 \begin{equation}\label{even-odd}
 F_1(\bx_{\diamond}')= F_1(\bx'), \qquad  F_2(\bx_{\diamond}')=-F_2(\bx'), \qquad  \bx' \in D,  \end{equation}
is called a (left)  generalized partial-slice function.
\end{definition}
\begin{remark}\label{remodd}
The condition in (\ref{even-odd}) means that $(F_1,F_2)$  is an even-odd pair in the $(p+2)$-th variable and so $F_2(\bx_p,0)=0$.
\end{remark}
Denote by $\mathcal{GS}(\Omega_{D})$ the set of  all induced  generalized partial-slice functions  on $\Omega_{D}$.  When the components $F_1,F_2$ are of class $C^k(D)$, we say that the generalized partial-slice function is of class $C^k(\Omega_{D})$ and  denote the set of all such functions by ${\mathcal{GS}}^{k}(\Omega_{D})$.
\begin{definition}\label{definition-GSR}
Let $f(\bx)=F_1(\bx')+\underline{\omega} F_2(\bx') \in {\mathcal{GS}}^{1}(\Omega_{D})$. The function $f$ is called generalized partial-slice monogenic  of type $(p,q)$ if  $F_1, F_2$ satisfy  the generalized Cauchy-Riemann equations
 \begin{eqnarray}\label{C-R}
 \left\{
\begin{array}{ll}
D_{\bx_p}  F_1(\bx')- \partial_{r} F_2(\bx')=0,
\\
 \overline{D}_{\bx_p}  F_2(\bx')+ \partial_{r} F_1(\bx')=0,
\end{array}
\right.
\end{eqnarray}
for all $\bx'\in D$.
\end{definition}
We denote by $\mathcal {GSR}(\Omega_D)$ the set of all generalized partial-slice monogenic functions {on $\Omega_D$}.
As before, the type $(p,q)$ will be omitted in the sequel.

\begin{remark} When $(p,q)=(0,n)$ a function of the form \eqref{genpslice} is called slice function, and if it satisfies system \eqref{C-R} it is a slice monogenic function.
\end{remark}
Now we recall a relationship between the set of functions  $\mathcal {GSM}$ and $\mathcal {GSR}$ defined in  p-symmetric domains, see Theorem 4.5 in \cite{Xu-Sabadini}. Notice that the second assertion follows from the Representation Formula.

\begin{theorem} \label{relation-GSR-GSM}

(i) For a p-symmetric domain $\Omega=\Omega_{D}$ with $\Omega  \cap \mathbb{R}^{p+1}= \emptyset$, it holds that $\mathcal {GSM}(\Omega) \supsetneqq \mathcal {GSR}(\Omega_{D})$.

(ii) For a p-symmetric domain $\Omega=\Omega_{D}$ with $\Omega  \cap \mathbb{R}^{p+1}\neq \emptyset$,  it holds  that $\mathcal {GSM}(\Omega) = \mathcal {GSR}(\Omega_{D})$.
\end{theorem}

\section{The Fueter-Sce  theorem}

To establish a  version of the Fueter-Sce  theorem  for generalized partial-slice monogenic functions, we need some lemmas. First, we recall a useful property for generalized partial-slice functions \cite{Xu-Sabadini}.
\begin{proposition}\label{relation}
If $f(\bx)=F_1(\bx')+\underline{\omega} F_2(\bx')\in {\mathcal{GS}}^{1}(\Omega_{D}) $, then the following two formulas hold on $\Omega_{D}\setminus \mathbb{R}^{p+1}$:

(i) $(D_{\bx}-{D_{\underline{\omega}})}f(\bx) =(1-q)\frac{ F_2(\bx')}{r}$;

(ii) $\Gamma f(\bx)=(q-1)\underline{\omega}  F_2(\bx')$,\\
where $\Gamma:=-\sum_{p+1\leq i<j\leq p+q} e_{i}e_{j}L_{ij}$ denotes the spherical Dirac operator on $\mathbb{R}^{q}$ with the angular momentum operators
$ L_{ij}=x_{i} \frac{\partial}{\partial x_{j}}-x_{j} \frac{\partial}{\partial x_{i}}$.
\end{proposition}

As it is well-known monogenic functions of axial type namely, with the terminology we adopted, a slice function which is also monogenic, satisfy a Vekua-type system. In the present setting generalized partial-slice  functions which are monogenic satisfy a Vekua-type system which, in fact, characterises these functions in the sense described in the next result.

  \begin{lemma}\label{GS-monogenic-lemma}
Let $D \subseteq \mathbb{R}^{p+2}$  be a domain invariant under the reflection with respect the $(p+2)$-th variable, and let $f(\bx)=F_1(\bx')+\underline{\omega} F_2(\bx')  \in {\mathcal{GS}}^{1}(\Omega_{D}) \cap C^{1}(\Omega_{D})$, where $\bx=\bx_p+r\underline{\omega}\in \Omega_{D}$, $\bx'=(\bx_p, r)\in D$. Then $f \in \mathcal {M}(\Omega_{D}) $ if and only if  the components $F_1,F_2$ of
$f$ satisfy the system
 \begin{eqnarray}\label{GS-is-Monogenic}
 \left\{
\begin{array}{ll}
D_{\bx_p}  F_1(\bx')- \partial_{r} F_2(\bx')=\frac{q-1}{r}F_2(\bx'),
\\
 \overline{D}_{\bx_p}  F_2(\bx')+ \partial_{r} F_1(\bx')=0,
\end{array}\quad \quad \bx'\in  D \setminus \mathbb{R}^{p+1}.
\right.
\end{eqnarray}
\end{lemma}

\begin{proof}
Let $f(\bx)=f(\bx_p+r\underline{\omega})=F_1(\bx')+\underline{\omega} F_2(\bx')  \in {\mathcal{GS}}^{1}(\Omega_{D})\cap C^{1}(\Omega_{D})$ and assume that $f\in\mathcal{M}(\Omega_D)$.
From  Proposition \ref{relation}, we deduce that  $f$ satisfies
\begin{eqnarray}\label{relation-slice-and-slice-monogenic}
D_{\bx} f(\bx) =D_{\bx_p}  F_1(\bx')- \partial_{r} F_2(\bx')+\frac{1-q}{r}F_2(\bx')+\underline{\omega}( \overline{D}_{\bx_p}  F_2(\bx')+ \partial_{r} F_1(\bx'))=0
\end{eqnarray}
 for all $\bx\in\Omega_{D}\setminus \mathbb{R}^{p+1}$. By the arbitrariness of $\underline{\omega}\in\mathbb{S}$ (for example by taking $\pm\underline{\omega}$) we deduce that the system \eqref{GS-is-Monogenic} is satisfied.

 Conversely, let us assume that the components  $F_1,F_2$ of $f(\bx)=F_1(\bx')+\underline{\omega} F_2(\bx')  \in {\mathcal{GS}}^{1}(\Omega_{D})\cap C^{1}(\Omega_{D})$, where $\bx=\bx_p+r\underline{\omega}$, $\bx'=(\bx_p, r)$, satisfy  the system (\ref{GS-is-Monogenic}). Then (\ref{relation-slice-and-slice-monogenic}) gives $D_{\bx} f=0$ on $\Omega_{D}\setminus \mathbb{R}^{p+1}$.  To conclude we observe that $\Omega_{D}\setminus \mathbb{R}^{p+1}$ is dense in $\Omega_D$ and since $f\in C^1(\Omega_D)$, we deduce that the condition $D_{\bx} f=0$ holds on the whole $\Omega_{D}$ and the conclusion follows.

\end{proof}

  \begin{remark}
Let $D \subseteq \mathbb{R}^{p+2}$  be a domain invariant under the reflection  with respect to the $(p+2)$-th variable, and let $f(\bx)=F_1(\bx')+\underline{\omega} F_2(\bx')  \in {\mathcal{GS}}^{1}(\Omega_{D})$. We note that $f \in \mathcal {M}(\Omega_{D}) $  implies that for all $\bx'=(\bx_{p},0)\in  D$,
 \begin{eqnarray}\label{GS-is-Monogenic-0}
 \left\{
\begin{array}{ll}
D_{\bx_p}  F_1(\bx')- q\partial_{r} F_2(\bx')=0,
\\
 \overline{D}_{\bx_p}  F_2(\bx')+ \partial_{r} F_1(\bx')=0.
\end{array}
\right.
\end{eqnarray}
In fact, recalling Remark \ref{remodd} and the assumption that $f\in {\mathcal{GS}}^{1}(\Omega_{D})$, and so in particular $F_2\in C^{1}({D})$, we have
$$\lim_{r\rightarrow0}\frac{1}{r}F_2(\bx_{p},r)=\lim_{r\rightarrow0}\frac{1}{r}(F_2(\bx_{p},r)-F_2(\bx_{p},0))=\partial_{r} F_2(\bx_{p},0),\quad  (\bx_{p},0)\in  D. $$
Thus taking the limit for $r\to 0$ in (\ref{relation-slice-and-slice-monogenic}),  we infer that
$$D_{\bx} f(\bx_{p}) =D_{\bx_p}  F_1(\bx_{p},0)-q \partial_{r} F_2(\bx_{p},0)  +\underline{\omega}( \overline{D}_{\bx_p}  F_2(\bx_{p},0)+ \partial_{r} F_1(\bx_{p},0)),$$
 and the statement follows.
\end{remark}

In the results below we shall use following notations for the Laplace operators in $\mathbb{R}^{p+q+1}$ and $\mathbb{R}^{p+1}$, respectively,
$$\Delta_{\bx}=\sum_{i=0}^{p+q} \partial_{x_i}^{2}, \qquad \quad \Delta_{\bx_p}=\sum_{i=0}^{p} \partial_{x_i}^{2}.$$
\begin{lemma}  \label{Fueter-Sce-Qian-lemma-1}
Let    $f(\bx)=F_1(\bx')+\underline{\omega} F_2(\bx') \in {\mathcal{GS}}^{2}(\Omega_{D})\cap C^{2}(\Omega_{D})$. Then it holds that for all $\bx\in\Omega_{D}$
\begin{eqnarray}\label{relation-harm}
\Delta_{\bx} f(\bx)= \Delta_{\bx'}F_1(\bx')+ \underline{\omega} \Delta_{\bx'}F_2(\bx')+ (q-1) ( (\frac{1}{r}\partial_{r}) F_1(\bx')+\underline{\omega} (\partial_{r}\frac{1}{r}) F_2(\bx') ),
\end{eqnarray}
where  $\Delta_{\bx'}=\Delta_{\bx_p}+\partial_{r}^{2}$ and the values of $(\frac{1}{r}\partial_{r}) F_1$, $(\partial_{r}\frac{1}{r}) F_2$ at  $\bx'=(\bx_{p},0)\in  D$   are  meant as
\begin{eqnarray}\label{relation-harm-000}
\lim_{r\rightarrow0} (\frac{1}{r}\partial_{r}) F_1(\bx')=\partial_{r}^{2}F_1(\bx_{p},0),
\end{eqnarray}
and
\begin{eqnarray}\label{relation-harm-0001}
\lim_{r\rightarrow0}(\partial_{r}\frac{1}{r}) F_2(\bx')=\frac{1}{2}\partial_{r}^{2}F_2(\bx_{p},0),
\end{eqnarray}
respectively.
\end{lemma}
\begin{proof}Let    $f(\bx)=F_1(\bx')+\underline{\omega} F_2(\bx') \in {\mathcal{GS}}^{2}(\Omega_{D})\cap C^{2}(\Omega_{D})$.
First, we point out that   the existence of limits in  (\ref{relation-harm-000}) and (\ref{relation-harm-0001}) depends heavily on  that $(F_1,F_2)$  is an even-odd pair in the $(p+2)$-th variable and then the two limits can be obtained with elementary arguments, e.g. with L'Hospital's rule in view of the fact that $F_1,F_2 \in  C^{2}(D)$.

The next step will prove (\ref{relation-harm}) for  $\bx\in \Omega_{D}\setminus \mathbb{R}^{p+1}$ so that the case  $\bx=(\bx_{p},0)\in\Omega_{D}$  follows immediately by  the  continuity. To this end, consider the decomposition
$$\Delta_{\bx}=\Delta_{\bx_p}+\Delta_{\underline{\bx}_q}, \quad \Delta_{\bx_p}=\sum_{i=0}^{p} \partial_{x_i}^{2}, \ \Delta_{\underline{\bx}_q}=\sum_{i=p+1}^{p+q} \partial_{x_i}^{2},$$
and recall that
$$\Delta_{\underline{\bx}_q}=\partial_{r}^{2}+\frac{q-1}{r}\partial_{r}+\frac{1}{r^{2}} \Delta_{\underline{\omega}},$$
where $\Delta_{\underline{\omega}}$ is (up to minus sign)  the Laplace-Beltrami operator on  $\mathbb{S}$ given by
$$\Delta_{\underline{\omega}} = (q-2)\Gamma -\Gamma ^{2}.$$
From the basic formulas
 $$L_{ij}F_2(\bx')=0,\ \Gamma F_1(\bx') =\Gamma F_2(\bx')=0,\ \Gamma \underline{\omega}=(q-1) \underline{\omega},$$
$$\Delta_{\underline{\omega}} \underline{\omega}= (q-2)(q-1) \underline{\omega}  -(q-1)\Gamma ^{2} \underline{\omega}=(1-q)  \underline{\omega}, $$
we have
$$\Delta_{\underline{\omega}}f(\bx) =(1-q)  \underline{\omega}F_2(\bx'),$$
and then
\begin{eqnarray}\label{Delta-f-F}
&&(\frac{q-1}{r}\partial_{r}+\frac{1}{r^{2}} \Delta_{\underline{\omega}})f(\bx)\notag
\\
 &=&\frac{q-1}{r}(\partial_{r}F_1(\bx') + \underline{\omega} \partial_{r}F_2(\bx'))+\frac{1}{r^{2}}(1-q)  \underline{\omega}F_2(\bx')\notag
 \\
&=& (q-1) ( \frac{1}{r}\partial_{r} F_1(\bx') +\underline{\omega} \partial_{r}  \frac{F_2(\bx')}{r} ).
\end{eqnarray}
Hence,
\begin{eqnarray*}
\Delta_{\bx}f(\bx)
 &=&\Delta_{\bx_{p}}f(\bx)+(\partial_{r}^{2}+\frac{q-1}{r}\partial_{r}+\frac{1}{r^{2}} \Delta_{\underline{\omega}})f(\bx)
 \\
&=&(\Delta_{\bx_{p}}+\partial_{r}^{2})f(\bx)+(\frac{q-1}{r}\partial_{r}+\frac{1}{r^{2}} \Delta_{\underline{\omega}})f(\bx)
\\ &=&\Delta_{\bx'}F_1(\bx')+ \underline{\omega} \Delta_{\bx'}F_2(\bx')+ (q-1) ( (\frac{1}{r}\partial_{r}) F_1(\bx')+\underline{\omega} (\partial_{r}\frac{1}{r}) F_2(\bx') ),
\end{eqnarray*}
as asserted.
\end{proof}

Before  giving the next lemma, we need to point out the important  fact   that all functions in $\mathcal{GSR}(\Omega_{D})$ are necessarily  real analytic.
 \begin{proposition}\label{slice-monogenic-real-analytic}
Let $f(\bx)=F_1(\bx')+\underline{\omega} F_2(\bx') \in {\mathcal{GSR}}(\Omega_{D})$. Then $F_1, F_2$  are   real analytic on $D$ and $f$ is real analytic on $\Omega_{D}$.
\end{proposition}
\begin{proof}Let $f(\bx)=F_1(\bx')+\underline{\omega} F_2(\bx') \in {\mathcal{GSR}}(\Omega_{D})$.
From the  Cauchy-Pompeiu formula \cite[Theorem 4.3]{Xu-Sabadini},  $F_1, F_2$  are   real analytic on $D$.
Recalling that $(F_1,F_2)$  is an even-odd pair in the $(p+2)$-th variable, from \cite[Theorem 1]{Whitney} and \cite[Theorem 2]{Whitney}, there exist real analytic functions $G_{1}, G_{2}$, respectively, such that
$$F_1(\bx_{p}, r)=G_{1}(\bx_{p}, r^{2}), \ F_2(\bx_{p}, r)=rG_{2}(\bx_{p}, r^{2}), \quad (\bx_{p}, r)\in D.$$
Then $$f(\bx)=f(\bx_p+\underline{\bx}_q)=G_{1}(\bx_{p}, r^{2})+ \underline{\bx}_q G_{2}(\bx_{p}, r^{2}), \quad r=|\underline{\bx}_q|.$$
Hence, $f$   is real analytic on $\Omega_{D}$. The proof is complete.
\end{proof}

 \begin{remark}\label{0-lim}
 Let $f(\bx)=F_1(\bx')+\underline{\omega} F_2(\bx') \in {\mathcal{GSR}}(\Omega_{D})$. Then by Proposition \ref{slice-monogenic-real-analytic}, we see that for $k \in \mathbb{N}$,
 \begin{equation}\label{r}
(\frac{1}{r}\partial_{r})^{k} F_1(\bx'), \ (\partial_{r}\frac{1}{r})^{k}F_2(\bx')
\end{equation}
 are well-posed for all $\bx'=(\bx_{p},r) \in  D$ with $r\neq0$.
  Reasoning as in the proof of Lemma \ref{Fueter-Sce-Qian-lemma-1}, we deduce that the following two limits exist
   \begin{equation}\label{r-lim}
\lim_{r\rightarrow0} (\frac{1}{r}\partial_{r})^{k} F_1(\bx'), \qquad \lim_{r\rightarrow0}(\partial_{r}\frac{1}{r})^{k}F_2(\bx').
\end{equation}
{\bf Assumption.} For the sake of simplicity, in Lemma \ref{Fueter-Sce-Qian-lemma-k}, and in Theorems \ref{Fueter-theorem}, \ref{Fueter-theorem-GSM}, \ref{CK-Fueter-relation} below,  we will use the convention that the values of  functions in (\ref{r}) at $(\bx_{p},0)\in  D$  are  meant as limits in (\ref{r-lim}), respectively.
  \end{remark}

\begin{lemma}  \label{Fueter-Sce-Qian-lemma-k}
Let $f(\bx)=F_1(\bx')+\underline{\omega} F_2(\bx') \in  \mathcal{GSR} (\Omega_{D})$. Then, for all $k \in \mathbb{N}\setminus\{0\}$,
\begin{equation}\label{k}
 \Delta_{\bx}^{k}f(\bx)= C_{q}(k)( (\frac{1}{r}\partial_{r})^{k}F_1(\bx')+\underline{\omega} (\partial_{r}\frac{1}{r})^{k}F_2(\bx') ),
\end{equation}
where $C_{q}(k)=(q-1)(q-3)\cdots(q-2k+1)$.
\end{lemma}
\begin{proof}  In view of Proposition \ref{slice-monogenic-real-analytic} and Remark \ref{0-lim}, both sides of  (\ref{k}) are well-defined.   As in the proof of Lemma \ref{Fueter-Sce-Qian-lemma-1},  we only prove (\ref{k}) for  $\bx\in \Omega_{D}\setminus \mathbb{R}^{p+1}$ since  the case of $\bx=(\bx_{p},0)\in\Omega_{D}$  follows immediately from the continuity. We follow  the reasoning of the proof of \cite[Lemma 1]{Qian-97} (see also \cite[Theorem 11.33]{Gurlebeck}),  and  use mathematical induction to prove (\ref{k}) for  $\bx\in \Omega_{D}\setminus \mathbb{R}^{p+1}$.

First,   observe that all generalized partial-slice monogenic functions are  necessarily slice-by-slice harmonic, that is to say $$f(\bx)=F_1(\bx')+\underline{\omega} F_2(\bx') \in {\mathcal{GSR}}(\Omega_{D}) \Rightarrow
 F_1(\bx') ,F_2(\bx') \in Ker \Delta_{\bx'},$$
 where $\Delta_{\bx'}=\Delta_{\bx_p}+\partial_r^2$.
Indeed, the conclusion    follows from   the identities
$$\Delta_{\bx_{p}}  F_1(\bx')
 =\overline{D}_{\bx_p} D_{\bx_p} F_1(\bx')=\overline{D}_{\bx_p}   \partial_{r} F_2 (\bx')=\partial_{r}\overline{D}_{\bx_p}   F_2(\bx') =-\partial_{r}^{2} F_1(\bx'),$$
$$\Delta_{\bx_{p}}  F_2(\bx')
 =D_{\bx_p} \overline{D}_{\bx_p}  F_2(\bx')=-D_{\bx_p}    \partial_{r} F_1(\bx') =-\partial_{r}D_{\bx_p}    F_1 (\bx')=-\partial_{r}^{2} F_2(\bx'),$$
and the fact that $F_1(\bx'), F_2(\bx')$ are  $C^{2}(D)$  by Proposition \ref{slice-monogenic-real-analytic}.\\
Hence by Lemma   \ref{Fueter-Sce-Qian-lemma-1}, we have
 $$\Delta_{\bx} f(\bx)=   (q-1) ( (\frac{1}{r}\partial_{r}) F_1(\bx')+\underline{\omega} (\partial_{r}\frac{1}{r}) F_2(\bx') ),$$
which shows the validity of (\ref{k}) for $k=1$.

Now suppose that (\ref{k}) holds for $k-1$, i.e.,
\begin{equation}\label{k-1}
 \Delta_{\bx}^{k-1}f(\bx)=C_{q}(k-1)( (\frac{1}{r}\partial_{r})^{k-1}F_1(\bx')+\underline{\omega} (\partial_{r}\frac{1}{r})^{k-1}F_2(\bx') ).
\end{equation}
To show (\ref{k})  from (\ref{k-1}), we set $A_{0}(\bx')=F_1(\bx'), B_{0}(\bx')=F_2(\bx')$, and, for $k\in \mathbb{N}\setminus\{0\}$,
$$A_{k}(\bx')=\frac{1}{r}\partial_{r}A_{k-1}(\bx') =    (\frac{1}{r}\partial_{r})^{k}F_1(\bx'), \quad
 B_{k}(\bx')=   \partial_{r} \frac{B_{k-1}}{r} =  (\partial_{r}\frac{1}{r})^{k}F_2(\bx'),$$
and  deduce by using also mathematical induction that, for $k\in \mathbb{N}$,
\begin{eqnarray}\label{inter-relation}
 \left\{
\begin{array}{ll}
D_{\bx_p}  A_{k}(\bx')- \partial_{r} B_{k}(\bx')=\frac{2k}{r}B_{k}(\bx'),
\\
 \overline{D}_{\bx_p} B_{k}(\bx') + \partial_{r} A_{k}(\bx')=0.
\end{array}
\right.
\end{eqnarray}
For $k=0$, the system coincides with the generalized Cauchy-Riemann equations, see Definition \ref{definition-GSR}. The step from $k-1$ to $k$ for (\ref{inter-relation}) is obtained as follows:
 \begin{eqnarray*}
 D_{\bx_p}  A_{k}(\bx')
&=&\frac{1}{r} \partial_{r} D_{\bx_p} A_{k-1}(\bx')\\
&=& \frac{1}{r} \partial_{r}  (\partial_{r} B_{k-1}(\bx')+\frac{2k-2}{r}B_{k-1}(\bx'))
 \\
&=&\partial_{r} B_{k}(\bx')+\frac{2k}{r}B_{k}(\bx'),
\end{eqnarray*}
and
 \begin{eqnarray*}
 \overline{D}_{\bx_p} B_{k}(\bx')
&=&\overline{D}_{\bx_p} ( \partial_{r} \frac{B_{k-1}(\bx')}{r}) \\
&=& \partial_{r} \frac{\overline{D}_{\bx_p} B_{k-1}(\bx')}{r}
 \\
&=&\partial_{r} \frac{- \partial_{r} A_{k-1}(\bx') }{r}\\
&=&- \partial_{r} A_{k}(\bx').
\end{eqnarray*}
 Now we can use (\ref{inter-relation}) to obtain (\ref{k})  from (\ref{k-1}). \\
 From (\ref{inter-relation}), we have
 \begin{eqnarray*}
 \Delta_{\bx_{p}}  A_{k-1}(\bx')
&=&\overline{D}_{\bx_p} D_{\bx_p}  A_{k-1}(\bx')\\
&=& \overline{D}_{\bx_p} (  \partial_{r}  +\frac{2k-2}{r})B_{k-1}(\bx')
 \\
&=&(  \partial_{r}  +\frac{2k-2}{r})\overline{D}_{\bx_p}B_{k-1}(\bx')\\
&=&(  \partial_{r}  +\frac{2k-2}{r}) (-\partial_{r}  A_{k-1}(\bx'))
\\
&=&-\partial_{r}^{2} A_{k-1}(\bx')-(2k-2)A_{k}(\bx'),
\end{eqnarray*}
and
 \begin{eqnarray*}
 \Delta_{\bx_{p}}  B_{k-1}(\bx')
&=&D_{\bx_p} \overline{D}_{\bx_p} B_{k-1}(\bx')\\
&=&D_{\bx_p} (- \partial_{r} A_{k-1}(\bx') )
 \\
&=&-\partial_{r}^{2} B_{k-1}(\bx')-(2k-2)\partial_{r} \frac{B_{k-1}(\bx')}{r}
\\
&=& -\partial_{r}^{2} B_{k-1}(\bx')-(2k-2) B_{k}(\bx'),
\end{eqnarray*}
then, combining with (\ref{Delta-f-F}), we obtain
\begin{eqnarray*}
\frac{\Delta_{\bx}^{k}f(\bx)}{C_{q}(k-1)}
&=&(\Delta_{\bx_{p}}+\partial_{r}^{2}+ \frac{q-1}{r}\partial_{r}+\frac{1}{r^{2}} \Delta_{\underline{\omega}}) ( A_{k-1}(\bx')+\underline{\omega} B_{k-1}(\bx') )
 \\
&=&\Delta_{\bx'} A_{k-1}(\bx')+\underline{\omega} \Delta_{\bx'} B_{k-1}(\bx') + (q-1)(\frac{1}{r}\partial_{r} A_{k-1}(\bx') +\underline{\omega} \partial_{r} \frac{B_{k-1}(\bx') }{r})
\\
&=&-(2k-2)(A_{k}(\bx') + \underline{\omega}  B_{k}(\bx')) +(q-1) (  A_{k}(\bx') +\underline{\omega}  B_{k}(\bx') )
\\
&=&(q-2k+1) (  A_{k}(\bx') +\underline{\omega}  B_{k}(\bx') ),
\end{eqnarray*}
which implies that
\begin{eqnarray*}
\Delta_{\bx}^{k}f(\bx)
&=&C_{q}(k-1)(q-2k+1)(  A_{k}(\bx') +\underline{\omega}  B_{k} (\bx'))
 \\
&=&C_{q}(k)( (\frac{1}{r}\partial_{r})^{k}F_1(\bx')+\underline{\omega} (\partial_{r}\frac{1}{r})^{k}F_2(\bx') ),
\end{eqnarray*}
as asserted.
\end{proof}

The previous results allow now to prove the Fueter-Sce  theorem for generalized partial-slice monogenic functions:
\begin{theorem}\label{Fueter-theorem}
Let     $f(\bx)=F_1(\bx')+\underline{\omega} F_2(\bx') \in \mathcal{GSR}(\Omega_{D})$.  Then, for odd $q\in \mathbb{N}$,
the function
 $$\tau_{q}f(\bx):=\frac{1}{(q-1)!!}\Delta_{\bx}^{\frac{q-1}{2}}f(\bx)=A(\bx')+\underline{\omega} B(\bx') $$
 where
 $$A(\bx') =  (\frac{1}{r}\partial_{r})^{\frac{q-1}{2}}F_1(\bx'),\quad  B(\bx')=  (\partial_{r}\frac{1}{r})^{\frac{q-1}{2}}F_2(\bx'),$$
  is monogenic in $\Omega_{D}$.
  \begin{proof}
 From  Lemmas  \ref{GS-monogenic-lemma} and \ref{Fueter-Sce-Qian-lemma-k}, we need only to   verify that the pair $(A, B)$ satisfies the equation (\ref{GS-is-Monogenic}), which is exactly (\ref{inter-relation}) with $k=\frac{q-1}{2}$.
 \end{proof}
 \end{theorem}

\begin{remark}
  Let $\mathcal{AM}(\Omega_D)$ be the special set of axially monogenic functions on $\Omega_D \subset \mathbb{R}^{p+q+1}$ defined by
 $$\mathcal{AM}(\Omega_D) = \mathcal{M} (\Omega_D) \cap \mathcal{GS}(\Omega_D). $$ Theorem \ref{Fueter-theorem} shows that the Fueter-Sce  map $\tau_q$ takes functions in $\mathcal{GSR}(\Omega_D)$ to $\mathcal{AM}(\Omega_D)$.
\end{remark}
 \begin{remark}
For $(p,q)=(0,n)$ with odd $n$, we obtain the well-known result, see \cite[Corollary, 3.10]{Colombo-Sabadini-Sommen-10} that
if $f$ is a slice monogenic function on an axially symmetric domain, then $\Delta^{\frac{n-1}{2}}f$ is  an axially monogenic function.
\end{remark}

The previous discussion can be reformulated in the following statement that we put in form of theorem since it provides a link between generalized partial-slice monogenic functions and classical monogenic functions.
\begin{theorem}\label{Fueter-theorem-GSM}
Let $\Omega\subseteq \mathbb{R}^{p+q+1}$ be a p-symmetric slice domain and let $f \in \mathcal{GSM}(\Omega)$. Then,  we can set $f(\bx)=F_1(\bx')+\underline{\omega} F_2(\bx')$ by Theorem \ref{Representation-Formula-SM},  and we have for odd $q\in \mathbb{N}$,
$$\tau_{q}f(\bx)=A(\bx')+\underline{\omega} B(\bx') \in \mathcal {M}(\Omega),$$
 where
 $$A(\bx') =  (\frac{1}{r}\partial_{r})^{\frac{q-1}{2}}F_1(\bx'),\quad  B(\bx')=  (\partial_{r}\frac{1}{r})^{\frac{q-1}{2}}F_2(\bx').$$
 \end{theorem}
\begin{proof}
It immediately follows from  Theorem  \ref{Fueter-theorem} and Theorem  \ref{relation-GSR-GSM}.
\end{proof}

\section{CK-extensions}

In this section we deepen the knowledge on CK-extension already given in \cite{Xu-Sabadini}, indeed in the sequel we shall need more refined results to work with the Radon transform.

We begin by observing that for each arbitrary but fixed $\eta\in\mathbb S$, generalized partial-slice monogenic functions can be seen as classical monogenic functions. Hence it is possible to obtain a type of  Cauchy-Kovalevskaya extension starting from a real analytic function defined on some domain in $\mathbb R^{p+1}$.

\begin{definition}[CK-extension]\label{Slice-Cauchy-Kovalevska-extension}
Let   $\Omega_{0}$ be a domain in  $\mathbb{R}^{p+1}$. The domain $\Omega_{0}^{\ast}$ is defined by
$$\Omega_{0}^{\ast}=\{ \bx_{p}+\underline{\bx}_{q}:\bx_{p}\in \Omega_{0}, \underline{\bx}_{q}\in \mathbb{R}^{q}\}.$$
Given a  real analytic function $f_{0}: \Omega_{0}  \to \mathbb{R}_{p+q}$, if there is a generalized partial-slice monogenic function $f^{\ast}$ in $\Omega \subseteq \Omega_{0}^{\ast}$ with $\Omega_{0} \subset \Omega$ and $f^{\ast}(\bx_{p})=f_0(\bx_{p})$.
The function $f^{\ast}$  is called the slice Cauchy-Kovalevskaya extension (CK-extension) of the real function $f_{0}$.
\end{definition}

\begin{theorem}\label{slice-Cauchy-Kovalevsky-extension}
Let   $\Omega_{0}$ be a domain in  $\mathbb{R}^{p+1}$ and  $f_{0}: \Omega_{0}  \to \mathbb{R}_{p+q}$ be a real analytic function.  Then the function given by
$$CK[f_{0}](\bx)= {\rm exp} (\underline{\bx}_q D_{\bx_p} ) f_{0}(\bx_p)
=\sum_{k=0}^{+\infty} \frac{1}{k!}( \underline{\bx}_q D_{\bx_p})^{k}f_{0}(\bx_p), \quad \bx=\bx_p+ \underline{\bx}_q,$$
is the solution of  the slice Cauchy-Kovalevskaya extension in a  p-symmetric slice domain $\Omega \subseteq \Omega_{0}^{\ast}$ with $\Omega_{0} \subset \Omega$.
\end{theorem}
\begin{proof}
 First, note that $CK[f_{0}]$ is well-defined since the series  $\sum_{k=0}^{+\infty} \frac{1}{k!}( \underline{\bx}_q D_{\bx_p})^{k}f_{0}(\bx_p)$ converges  in a   p-symmetric slice domain $\Omega \subseteq \Omega_{0}^{\ast}$ with $\Omega_{0} \subset \Omega$. To see this,
we consider its restriction to $\Omega_{\underline{\omega}}$ for each $\underline{\omega}\in \mathbb{S}$, i.e.,
 $$CK[f_{0}](\bx_p+ r \underline{\omega})=\sum_{k=0}^{+\infty} \frac{r^{k}}{k!}( \underline{\omega} D_{\bx_p})^{k}f_{0}(\bx_p).$$
Recall  that real analytic functions possess absolutely converging Taylor series representations at every point
$$f_{0}(\bx_p+\boldsymbol{h})=\sum_{k=0}^{+\infty} \sum_{|\mathrm{k}|=k} \frac{\boldsymbol{h}^{\mathrm{k}}}{\mathrm{k}!}
  \partial_{\mathrm{k}} f_{0}(\bx_p),\quad \bx_{p}\in  \Omega_{0},$$
  absolutely convergent  for $|\boldsymbol{h}|$ small enough. Let  $\boldsymbol{h}$ be the $(p+1)$-tuple $(h,...,h)$ with $h>0$; we have $\boldsymbol{h}^{\mathrm{k}}=h^{k}$ and
  $$\sum_{k=0}^{+\infty} h^{k}\sum_{|\mathrm{k}|=k} \frac{1}{\mathrm{k}!}|\partial_{\mathrm{k}} f_{0}(\bx_p)|<+\infty.$$
From the basic fact
\begin{equation}\label{Cauchy-Kovalevska-extension-odd-even-norm}
 |( \underline{\omega} D_{\bx_p})^{k}f_{0}(\bx_p)|=
\Big|\sum_{i_1,\ldots, i_k=0}^p s_{i_1\ldots i_k} e_{i_1}\ldots e_{i_k}\frac{\partial^k}{\partial_{x_{i_1}}\ldots \partial_{x_{i_k}}}f_0(\bx_p) \Big|,\quad \end{equation}
where $s_{i_1\ldots i_k}= \pm1,$ we infer that
$$|( \underline{\omega} D_{\bx_p})^{k}f_{0}(\bx_p)| \leq \sum_{i_1,\ldots, i_k=0}^p \Big|\frac{\partial^k}{\partial_{x_{i_1}}\ldots \partial_{x_{i_k}}}f_0(\bx_p)\Big|= k!\sum_{|\mathrm{k}|=k} \frac{1}{\mathrm{k}!}| \partial_{\mathrm{k}}f_{0}(\bx_p)|.$$
Thus,
$$\sum_{k=0}^{+\infty} \frac{r^{k}}{k!}|( \underline{\omega} D_{\bx_p})^{k}f_{0}(\bx_p)|
\leq\sum_{k=0}^{+\infty} r^{k} \sum_{|\mathrm{k}|=k} \frac{1}{\mathrm{k}!}| \partial_{\mathrm{k}}f_{0}(\bx_p)|.$$
We deduce that, for any $\underline{\omega}\in \mathbb{S}$ and any compact $K\subset  \Omega_{0}$, there exists some constant $h_{K}>0$ such that the series $\sum_{k=0}^{+\infty} \frac{r^{k}}{k!}( \underline{\omega} D_{\bx_p})^{k}f_{0}(\bx_p)$ converges  for $K \times (-h_{K},h_{K})$.
Therefore, the series $\sum_{k=0}^{+\infty} \frac{1}{k!}( \underline{\bx}_q D_{\bx_p})^{k}f_{0}(\bx_p)$   converges normally on $$\Omega=\bigcup_{K\subset  \Omega_{0}} K \times \{\underline{\bx}_{q}\in \mathbb{R}^{q}: |\underline{\bx}_{q}|<h_{K}\},$$
which is clearly a p-symmetric and slice domain in $\Omega_{0}^{\ast}$.

Now we show that $CK[f_{0}]\in \mathcal{GSM}^{L}(\Omega)$. To see this, fix $\underline{\omega}\in \mathbb{S}$ and then we have, for $\bx_p+ r \underline{\omega}\in \Omega_{\underline{\omega}}$,
\begin{eqnarray*}
 (\underline{\omega}\partial_{r})  CK[f_{0}](\bx_p+ r \underline{\omega})
 &=&\underline{\omega} \sum_{k=0}^{+\infty} \frac{r^{k}}{k!}( \underline{\omega} D_{\bx_p})^{k+1}f_{0}(\bx_p)
 \\
&=&-D_{\bx_p}\sum_{k=0}^{+\infty} \frac{r^{k}}{k!}( \underline{\omega} D_{\bx_p})^{k}f_{0}(\bx_p)
 \\
&=&-D_{\bx_p}CK[f_{0}](\bx_p+ r \underline{\omega}).
\end{eqnarray*}
Hence,
$$D_{\underline{\omega}} CK[f_{0}]_{\underline{\omega}}=0.$$

Finally, combined with the obvious fact  $CK[f_{0}](\bx_p)=f_{0}(\bx_p)$, we get the assertion.

\end{proof}

\begin{remark}\label{Cauchy-Kovalevska-extension-odd-even}
In fact, formula   (\ref{Cauchy-Kovalevska-extension-odd-even-norm})  depends on the facts that $\Delta_{\bx_p}$ is real operator  and the identities
$$( \underline{\omega} D_{\bx_p})^{2}=\underline{\omega}\, \underline{\omega} \overline{D}_{\bx_p}D_{\bx_p}=  -\Delta_{\bx_p}.$$
which   imply also that  $CK[f_{0}]$ has the following  decomposition
\begin{eqnarray*}
 & & CK[f_{0}](\bx_p+ r \underline{\omega})\\
 &=&\sum_{k=0}^{+\infty} \frac{r^{2k}}{(2k)!} \underline{\omega}^{2k} \Delta_{\bx_p} ^{k} f_{0}(\bx_p)
+\sum_{k=0}^{+\infty} \frac{r^{2k+1}}{(2k+1)!}\underline{\omega}^{2k+1}  \Delta_{\bx_p} ^{k} D_{\bx_p} f_{0}(\bx_p)
 \\
&=&\sum_{k=0}^{+\infty} \frac{r^{2k}}{(2k)!}  (-\Delta_{\bx_p})^{k} f_{0}(\bx_p)
+\underline{\omega} D_{\bx_p} \sum_{k=0}^{+\infty} \frac{r^{2k+1}}{(2k+1)!}(-\Delta_{\bx_p}) ^{k}f_{0}(\bx_p).
\end{eqnarray*}
\end{remark}

\begin{theorem}Let $f_{0}:  \mathbb{R}^{p+1} \to \mathbb{R}_{p+q}$ be a real analytic function. Then $CK[f_{0}]$  is the unique extension of $f_{0}$ to $\mathbb{R}^{p+q+1}$ preserving generalized partial-slice monogenicity.
 In particular, for $f_{0}: \mathbb{R}^{p+1} \to\mathbb{R}$, we have $CK[f_{0}](\bx)=F_1(\bx')+\underline{\omega} F_2(\bx') \in \mathcal{GSM}^{L}(\mathbb{R}^{p+q+1})$ with $F_1(\bx')\in \mathbb{R}, F_2(\bx')\in \mathrm{span}_{\mathbb{R}}\{1,e_{1},\ldots,e_p \}$ and
$$CK[f_{0}](\bx_p +r \underline{\omega}) (\overline{D}_{\bx_p}+ \underline{\omega}\partial_{r})=0.$$
\end{theorem}
\begin{proof}
In view of  Theorem  \ref{Identity-theorem},  Theorem \ref{slice-Cauchy-Kovalevsky-extension} and Remark \ref{Cauchy-Kovalevska-extension-odd-even}, it remains to  prove the last statement $CK[f_{0}](\bx_p +r \underline{\omega}) (\overline{D}_{\bx_p}+ \underline{\omega}\partial_{r})=0.$ To see this, keeping in mind of $f_{0}$ being real-valued, we have
\begin{eqnarray*}
   CK[f_{0}](\bx_p+ r \underline{\omega}) (\underline{\omega}\partial_{r})
 &=&\Big( \sum_{k=0}^{+\infty} \frac{r^{k}}{k!}( \underline{\omega} D_{\bx_p})^{k+1}f_{0}(\bx_p)\Big) \underline{\omega}
 \\
  &=&\sum_{k=0}^{+\infty} \frac{r^{k}}{k!} \Big(f_{0}(\bx_p) ( \underline{\omega} D_{\bx_p})^{k+1}\Big) \underline{\omega}
 \\
  &=&\sum_{k=0}^{+\infty} \frac{r^{k}}{k!} \Big(f_{0}(\bx_p) ( \underline{\omega} D_{\bx_p})^{k}\Big) \underline{\omega} D_{\bx_p}\underline{\omega}
 \\
 &=&\sum_{k=0}^{+\infty} \frac{r^{k}}{k!} \Big( ( \underline{\omega} D_{\bx_p})^{k}f_{0}(\bx_p)\Big) \overline{D}_{\bx_p} \underline{\omega}^{2}
 \\
&=&-CK[f_{0}](\bx_p+ r \underline{\omega})\overline{D}_{\bx_p},
\end{eqnarray*}
which completes the proof.
\end{proof}

In the following result, we establish   the generalized CK-extension for monogenic functions, which is a   modified version of \cite[Theorem 5.1.1]{Delanghe-Sommen-Soucek}.
\begin{theorem}[generalized CK-extension]\label{generalized-CK-extension}
Let  $\Omega_{0}\subseteq \mathbb{R}^{p+1}$ be a domain and  consider the real analytic  function $f_{0}(\bx_{p}): \Omega_{0}\rightarrow \mathbb{R}_{p+q}$. Then there exists a unique real analytic functions sequence $\{f_{k}(\bx_{p})\}_{k=1}^{\infty} $ such that the series
\begin{equation}\label{Cauchy-Kovalevska-extension-sum}
f(\bx) =GCK[f_{0}](\bx)= \sum_{k=0}^{+\infty} \underline{\bx}_{q}^{k}f_{k}(\bx_p) \end{equation}
is convergent in a p-symmetric  slice domain $\Omega \subseteq \Omega_{0}^{\ast}$ with $\Omega_{0} \subset \Omega$ and its sum $f$ is monogenic in $\Omega$.
Furthermore,  the function $f$ can be written in a formal way as
\begin{equation}\label{Cauchy-Kovalevska-extension-expression}
 f(\bx) =\Gamma(\frac{q}{2})( \frac{r\sqrt{\Delta_{\bx_{p}}}}{2})^{-\frac{q}{2}} \big[ \frac{r\sqrt{\Delta_{\bx_{p}}}}{2} J_{\frac{q}{2}-1}(r\sqrt{\Delta_{\bx_{p}}}) + \frac{\underline{\bx}_{q}D_{\bx_{p}}}{2} J_{\frac{q}{2}}(r\sqrt{\Delta_{\bx_{p}}}) \big] f_{0}(\bx_{p}), \end{equation}
where $r=|\underline{\bx}_{q}|$, $\sqrt{\Delta_{\bx_{p}}}$ denotes the square root of the Laplacian $\Delta_{\bx_{p}}$ (of which only even powers
occur in the resulting series),  and $J_{n}$ is the Bessel function of the first kind of order $n$ given by
$$J_{n}(x)= \sum_{k=0}^{+\infty}\frac{(-1)^{k}(x/2)^{n+2k}}{k!\Gamma(n+k+1)}.$$
\end{theorem}
\begin{proof}From the following  basic identities, for all $k\in \mathbb{N}$,
$$D_{\underline{\bx}_q} \underline{\bx}_q^{2k} =-2k \underline{\bx}_q^{2k-1},$$
$$D_{\underline{\bx}_q}\underline{\bx}_q^{2k+1} =-(2k+q) \underline{\bx}_q^{2k},$$
$$D_{\bx_p} \underline{\bx}_q^{2k} =\underline{\bx}_q^{2k}D_{\bx_p},  $$
$$D_{\bx_p}\underline{\bx}_q^{2k+1}  =\underline{\bx}_q^{2k+1}\overline{D}_{\bx_p},$$
where $D_{\bx_p}$ and  $\underline{\bx}_q^{k}$ in last two formulae are both considered as operators on functions,
and applying the operator $(D_{\bx_p} +D_{\underline{\bx}_q})$ to the  monogenic sum $f(\bx)$ in (\ref{Cauchy-Kovalevska-extension-sum}), we have  for all $k\in \mathbb{N}$
\begin{eqnarray*}\label{inter-relation-f}
 \left\{
\begin{array}{ll}
(2k+q) f_{2k+1}(\bx_{p})=D_{\bx_p} f_{2k}(\bx_{p}),
\\
(2k+2) f_{2k+2}(\bx_{p})=\overline{D}_{\bx_p} f_{2k+1}(\bx_{p}).
\end{array}
\right.
\end{eqnarray*}
From these iterative relations above,  we infer that
\begin{eqnarray*}
f_{2k}(\bx_{p})& =&\frac{1}{2k}\overline{D}_{\bx_p} f_{2k-1}(\bx_{p})\\
 &=&\frac{1}{2k(2k+q-2)}   \overline{D}_{\bx_p} D_{\bx_p} f_{2k-2}(\bx_{p})
 \\
&=&\frac{1}{2k(2k+q-2)}   \Delta_{\bx_p} f_{2k-2}(\bx_{p})\\
&=&\cdots
\\
&=&\frac{1}{(2k)!!(2k+q-2)(2k+q-4)\cdots q}   \Delta_{\bx_p}^{k} f_{0}(\bx_{p})
\\
&=&\frac{\Gamma(\frac{q}{2})}{2^{2k}k!\Gamma(k+\frac{q}{2})}   \Delta_{\bx_p}^{k} f_{0}(\bx_{p}),
\end{eqnarray*}
and
\begin{eqnarray*}
 f_{2k+1}(\bx_{p})& =&\frac{1}{2k+q} D_{\bx_p} f_{2k}(\bx_{p})\\
 &=& \frac{1}{2k(2k+q)}  D_{\bx_p} \overline{D}_{\bx_p} f_{2k-1}(\bx_{p})
 \\
&=&\frac{1}{2k(2k+q)}  \Delta_{\bx_p} f_{2k-1}(\bx_{p})\\
&=&\cdots
\\
&=&\frac{1}{(2k)!!(2k+q)(2k+q-2)\cdots (q+2)}   \Delta_{\bx_p}^{k} f_{1}(\bx_{p}),
\\
&=&\frac{1}{(2k)!!(2k+q)(2k+q-2)\cdots (q+2)q}   \Delta_{\bx_p}^{k} D_{\bx_p} f_{0}(\bx_{p})
\\
&=&\frac{\Gamma(\frac{q}{2})}{2^{2k+1}k!\Gamma(k+\frac{q}{2}+1)} D_{\bx_p}  \Delta_{\bx_p}^{k} f_{0}(\bx_{p}).
\end{eqnarray*}
Consequently,  the expression of $f$ is given by
$$ f(\bx)=\sum_{k=0}^{+\infty} \frac{\Gamma(\frac{q}{2})(-1)^{k}r^{2k}}{2^{2k}k!\Gamma(k+\frac{q}{2})}   \Delta_{\bx_p}^{k} f_{0}(\bx_{p})+
\underline{\bx}_q  D_{\bx_p}  \sum_{k=0}^{+\infty} \frac{\Gamma(\frac{q}{2})(-1)^{k}r^{2k}}{2^{2k+1}k!\Gamma(k+\frac{q}{2}+1)} \Delta_{\bx_p}^{k} f_{0}(\bx_{p}),$$
that is to say $f$ takes the form of (\ref{Cauchy-Kovalevska-extension-expression}). Finally, following the classical method as in \cite[Theorem 5.1.1]{Delanghe-Sommen-Soucek}, one can show verbatim that the series converges in a  p-symmetric slice domain $\Omega \subseteq \Omega_{0}^{\ast}$ with $\Omega_{0} \subset \Omega$. The proof is complete.
\end{proof}

Denote by $\mathcal{A}(\Omega_{0})$ the set of all $\mathbb{R}_{p+q}$-valued  real analytic  functions defined in $ \Omega_{0} \subseteq \mathbb{R}^{p+1}$.
The slice monogenic versions of the next theorem was proved in the paper \cite{De-Diki-Adan}.

\begin{theorem}\label{CK-Fueter-relation}
Let $\Omega\subseteq \mathbb{R}^{p+q+1}$ be a p-symmetric slice domain and $f:\Omega\rightarrow \mathbb{R}_{p+q}$ be a  generalized partial-slice monogenic function. Set $f(\bx)=F_1(\bx')+\underline{\omega} F_2(\bx')$ and $\Omega_{0}=\Omega\cap \mathbb{R}^{p+1}$.   Then,  for odd $q\in \mathbb{N}$,
we have
 $$\tau_{q}f(\bx) = GCK [(\frac{1}{r}\partial_{r})^{\frac{q-1}{2}}F_1(\bx_{p},0 )]
 =\frac{1}{(q-2)!!} GCK [ (-\Delta_{\bx_{p}})^{\frac{q-1}{2}}f(\bx_{p}) ],$$
which gives the  commutative diagram
$$\xymatrix{
 \mathcal{A}(\Omega_{0}) \ar[d]_{\frac{1}{(q-2)!!}(-\Delta_{\bx_{p}})^{\frac{q-1}{2}}} \ar[r]^{CK} & \mathcal{GSM}(\Omega) \ar[d]^{\tau_{q}} \\
  \mathcal{A}(\Omega_{0}) \ar[r]^{GCK} & \mathcal{AM}(\Omega).  }$$
 \end{theorem}
 \begin{proof}
By Theorem \ref{Fueter-theorem-GSM}, we have
  $$\tau_{q}f(\bx)=A(\bx')+\underline{\omega} B(\bx') \in \mathcal {M}(\Omega),$$
 where
 $$A(\bx') =  (\frac{1}{r}\partial_{r})^{\frac{q-1}{2}}F_1(\bx'),\quad  B(\bx')=  (\partial_{r}\frac{1}{r})^{\frac{q-1}{2}}F_2(\bx').$$
 Since $B(\bx_{p}, r)$ is odd in the variable $r$, we  obtain that
 $$\tau_{q}f(\bx)\mid_{\underline{\bx}_{q}=0}=A(\bx_{p},0) =  (\frac{1}{r}\partial_{r})^{\frac{q-1}{2}}[F_1] (\bx_{p},0).$$
 From the generalized CK-extension (Theorem \ref{generalized-CK-extension}), the   monogenic function $\tau_{q}f$ is completely determined by its restriction to $\mathbb{R}^{p+1}$, which shows
$$\tau_{q}f(\bx)= GCK [(\frac{1}{r}\partial_{r})^{\frac{q-1}{2}}[F_1](\bx_{p},0 )](\bx).$$

On the other hand, by Remark \ref{Cauchy-Kovalevska-extension-odd-even} and  Theorem \ref{Identity-theorem}, it holds that
$$F_1(\bx')= \sum_{k=0}^{+\infty} \frac{r^{2k}}{(2k)!}  (-\Delta_{\bx_p})^{k} f(\bx_{p})=
\sum_{k=0}^{+\infty} \frac{r^{2k}}{(2k)!}  (-\Delta_{\bx_p})^{k} [F_1](\bx_{p},0 ),$$
which gives that for all $l\in \mathbb{N}$
$$(\frac{1}{r}\partial_{r})^{l}[F_1](\bx')=\sum_{k=l}^{+\infty} \frac{(2k)(2k-2)\cdots (2k-2l+2)}{(2k)!}  r^{2k-2l}(-\Delta_{\bx_p})^{k} [F_1](\bx_{p},0 ).$$
In particular,
$$(\frac{1}{r}\partial_{r})^{l}[F_1](\bx_{p},0)=  \frac{(2l)!!}{(2l)!} (-\Delta_{\bx_p})^{l} [F_1](\bx_{p},0 )
=\frac{1}{(2l-1)!!} (-\Delta_{\bx_p})^{l} f(\bx_{p}),$$
in which letting $l=\frac{q-1}{2}$ we obtain
$$(\frac{1}{r}\partial_{r})^{\frac{q-1}{2}}[F_1](\bx_{p},0)= \frac{1}{(q-2)!!} (-\Delta_{\bx_p})^{\frac{q-1}{2}} f(\bx_{p}).$$
The proof is complete.
\end{proof}
\section{The Radon transform and its dual}
As we discussed in Section 3, the Fueter-Sce  mapping theorem provides a link between generalized partial-slice monogenic functions and monogenic functions, see Theorem \ref{Fueter-theorem-GSM}. However, like in the standard slice monogenic case, see \cite{Colombo-Sabadini-Soucek-15}, we have another transform providing this link, namely the dual Radon transform.

We recall that given  $f:\mathbb{R}^{p+q+1} \longrightarrow  \mathcal{C}l_{p+q}$, the  Radon transform $R$ of   $f$  is defined by
$$R[f](\bx_p,r,\underline{\omega})= \int_{L(\underline{\omega},r)} f(\bx_p,\underline{\bx}_q) d\sigma(\underline{\bx}_q)$$
whenever the integral exists, where $ d\sigma$  is  the Lebesgue measure on the hyperplane $$L(\underline{\omega},r)=\{\underline{\bx}_q \in   \mathbb{R}^q :\langle\underline{\bx}_q,\underline{\omega} \rangle=r \}.$$
The dual Radon transform $\check{R}$ of a continuous function $f$ is defined by
$$\check{R} [f](\bx_p,\underline{\bx}_q)=\frac{1}{A_{q}}\int_{\mathbb{S}} f(\bx_p,\langle\underline{\bx}_q,\underline{\eta} \rangle\underline{\eta}) dS(\underline{\eta})$$
where $\langle \cdot, \cdot\rangle$ is the scalar product on $\mathbb{R}^{q}$, $dS$  is the area element on the sphere $\mathbb{S}$ and  $A_{q}$ is the area of $\mathbb{S}$.

For more details and results on the  Radon transform we refer the reader to \cite{Helgason}.  Note  that the variable $\bx_p\in \mathbb{R}^{p+1}$ plays a role of parameter with respect to  transforms   $R$ and $\check{R}$.  We have the following immediate result.

\begin{lemma}\label{slice-monogenic-relation}
For real  analytic functions $f:\Omega\subseteq\mathbb{R}^{p+q+1} \longrightarrow  \mathcal{C}l_{p+q}$ with suitable integrability property, we have
 $$\check{R}[D_{\underline{\omega}}f]=D_{\bx}(\check{R}[f]),$$
and $$R[D_{\bx}f]=D_{\underline{\omega}}(R[f]).$$
\end{lemma}
\begin{proof}
The statements are proved reasoning as in the proof of \cite[Lemma 4.2 ]{Colombo-Sabadini-Soucek-15} and \cite[Lemma 5.3]{Colombo-Sabadini-Soucek-15}.
\end{proof}

To obtain the relation of the generalized CK-extension and the dual Radon transform,  we define a Cauchy-Kovalevskaja extension for a fixed direction.
\begin{definition}[CK-extension for a fixed direction]\label{Cauchy-Kovalevska-extension-fixed}
Let $\underline{\eta}\in \mathbb{S}$, $\Omega_{0}\subseteq \mathbb{R}^{p+1}$ be a domain, and $f_{0}: \Omega_{0}  \to \mathbb{R}_{p+q}$ be a real analytic function. Define
$$CK[f_{0}, \underline{\eta}](\bx)= {\rm exp} ( \langle \underline{\eta},\underline{\bx}_{q}\rangle \underline{\eta} D_{\bx_p} ) f_{0}(\bx_p)
=\sum_{k=0}^{+\infty} \frac{ \langle \underline{\eta},\underline{\bx}_{q}\rangle  ^{k}}{k!}( \underline{\eta} D_{\bx_p} )^{k}f_{0}(\bx_p).$$
\end{definition}

Using the same method of proving  Theorem \ref{slice-Cauchy-Kovalevsky-extension}, it can be shown that  the series in Definition \ref{Cauchy-Kovalevska-extension-fixed} converges   in a p-symmetric slice domain $\Omega\subseteq \Omega_{0}^{\ast}$ with $\Omega_{0} \subset \Omega$.

 \begin{lemma} \label{Cauchy-Kovalevska-extension-fixed-lemma}
Let $\underline{\eta}\in \mathbb{S}$, $\Omega_{0}\subseteq \mathbb{R}^{p+1}$ be a domain, and $f_{0}: \Omega_{0}  \to \mathbb{R}_{p+q}$ be a real analytic function. Then
 $$\partial_{\underline{\eta}}  CK[f_{0}, \underline{\eta}](\bx) =0, \  D_{\bx} CK[f_{0}, \underline{\eta}](\bx)=0, \quad  \bx\in\Omega,$$
where $\partial_{\underline{\eta}}= D_{\bx_p}+ \underline{\eta}  \langle \underline{\eta},D_{\underline{\bx}_{q}}\rangle$ and $\Omega$ is a p-symmetric slice  domain $\Omega\subseteq \Omega_{0}^{\ast}$ with $\Omega_{0} \subset \Omega$.
\end{lemma}
  \begin{proof}
First, note that the CK-extension for a fixed direction has a decomposition of that
\begin{eqnarray*}
&&CK[f_{0}, \underline{\eta}](\bx)\\
 &=&\sum_{k=0}^{+\infty} \frac{ \langle \underline{\eta},\underline{\bx}_{q}\rangle  ^{2k}}{(2k)!}(-\Delta_{\bx_p} )^{k}f_{0}(\bx_p)+ \underline{\eta} D_{\bx_p}\sum_{k=0}^{+\infty} \frac{ \langle \underline{\eta},\underline{\bx}_{q}\rangle  ^{2k+1}}{(2k+1)!}(-\Delta_{\bx_p} )^{k}f_{0}(\bx_p)
 \\
&=&\sum_{k=0}^{+\infty} \frac{( \langle \underline{\eta},\underline{\bx}_{q}\rangle \underline{\eta}) ^{2k}}{(2k)!} \Delta_{\bx_p}^{k}f_{0}(\bx_p)+ \sum_{k=0}^{+\infty} \frac{ (\langle \underline{\eta},\underline{\bx}_{q}\rangle\underline{\eta})^{2k+1}}{(2k+1)!}\Delta_{\bx_p}^{k}D_{\bx_p}f_{0}(\bx_p).
\end{eqnarray*}
From the basic  formula
$$\langle \underline{\eta},D_{\underline{\bx}_{q}}\rangle   \langle \underline{\eta},\underline{\bx}_{q}\rangle^{k} =
 k \langle \underline{\eta},\underline{\bx}_{q}\rangle^{k-1}, $$
we obtain
\begin{eqnarray*}\underline{\eta}  \langle \underline{\eta},D_{\underline{\bx}_{q}}\rangle    CK[f_{0}, \underline{\eta}]
 &=&\underline{\eta}  \sum_{k=1}^{+\infty} \frac{\langle \underline{\eta},\underline{\bx}_{q}\rangle^{k-1}}{(k-1)!}(\underline{\eta}  D_{\bx_p})^{k}f_{0}(\bx_p)
 \\
&=&\underline{\eta}(\underline{\eta}  D_{\bx_p} )\sum_{k=1}^{+\infty} \frac{\langle \underline{\eta},\underline{\bx}_{q}\rangle^{k-1}}{(k-1)!}(\underline{\eta}  D_{\bx_p})^{k-1}f_{0}(\bx_p)
 \\
&=&-D_{\bx_p}CK[f_{0}, \underline{\eta}],
\end{eqnarray*}
which shows the first stated result $\partial_{\underline{\eta}}CK[f_{0}, \underline{\eta}]=0.$

To see the second result, we write
$$ D_{\bx}=D_{\bx_p}+ D_{\underline{\bx}_q}=D_{\bx_p}+ \sum_{l=1}^{q} I_{l}  \langle I_{l} ,D_{\underline{\bx}_{q}}\rangle,$$
where $\{I_{1}=  \underline{\eta}, I_2,..., I_{q}\}$ is an orthonormal basis of $\mathbb{R}^{q}$.\\
From the identity
$$  \langle I_{l} ,D_{\underline{\bx}_{q}}\rangle CK[f_{0}, \underline{\eta}]=  \langle I_{l},\underline{\eta}\rangle \underline{\eta} D_{\bx_p} CK[f_{0}, \underline{\eta}]=0,\quad l=2,3,...,q,$$
we infer that
$$ D_{\bx} CK[f_{0}, \underline{\eta}] =\partial_{\underline{\eta}}  CK[f_{0}, \underline{\eta}] =0,$$
which finishes the proof.
\end{proof}
The slice monogenic version of next result was originally proved in \cite{De-Diki-Adan}. Here we prove its counterpart for generalized partial-slice monogenic functions.
\begin{theorem}\label{GCK-Dual-relation}
Let $\Omega_{0}$ be a domain in  $\mathbb{R}^{p+1}$ and $f_{0}: \Omega_{0}\to \mathbb{R}_{p+q}$ be a real analytic function. Then
$$GCK[f_{0}](\bx)=\frac{1}{A_{q}}\int_{\mathbb{S}} CK[f_{0}, \underline{\eta}](\bx) dS(\underline{\eta})=\check{R} [f](\bx), \quad  \bx\in\Omega,$$
where $f=CK[f_{0}]$ is the slice CK-extension in a  p-symmetric slice domain    $\Omega\subseteq \Omega_{0}^{\ast}$ with $\Omega_{0} \subset \Omega$. That is to say $GCK=\check{R} \circ CK$, which gives the  commutative diagram
$$\xymatrix{
\mathcal{ A}(\Omega_{0}) \ar[dr]_{GCK} \ar[r]^{CK}
                & \mathcal{GSM}(\Omega) \ar[d]^{\check{R}}  \\
                & \mathcal{AM}(\Omega).           }$$
\end{theorem}

\begin{proof}
From the proof of Lemma \ref{Cauchy-Kovalevska-extension-fixed-lemma}, we consider
$$g(\bx_{p},\langle \underline{\eta},\underline{\bx}_{q}\rangle \underline{\eta})=CK[f_{0}, \underline{\eta}](\bx),\quad  \bx\in\Omega.$$
From the Funk-Hecke formulas \cite[p. 172]{Delanghe-Sommen-Soucek}, it holds that
$$\int_{\mathbb{S}}  \langle \underline{\eta},\underline{\bx}_{q}\rangle  ^{2k} dS(\underline{\eta})= C(k,q) |\underline{\bx}_{q}|^{2k},$$
and
$$\int_{\mathbb{S}}  \underline{\eta} \langle \underline{\eta},\underline{\bx}_{q}\rangle  ^{2k+1} dS(\underline{\eta})=C(k+1,q)
|\underline{\bx}_{q}|^{2k}\underline{\bx}_{q},$$
where $C(k,q)=\int_{0}^{\pi} \cos^{2k}\theta \sin^{q-1}\theta d\theta$.\\
From these two equalities above, we infer that   the integral $g$ on ${\mathbb{S}}$ takes the form of
\begin{eqnarray*}
\int_{\mathbb{S}} g(\bx_p,\langle\underline{\bx}_q,\underline{\eta} \rangle\underline{\eta}) dS(\underline{\eta})
&=&\sum_{k=0}^{+\infty} \frac{ C(k,q)|\underline{\bx}_{q}|^{2k}}{(2k)!}(-\Delta_{\bx_p} )^{k}f_{0}(\bx_p)\\
& &
  +\underline{\bx}_{q}\sum_{k=0}^{+\infty} \frac{ C(k+1,q)|\underline{\bx}_{q}|^{2k}}{(2k+1)!}D_{\bx_p}(-\Delta_{\bx_p} )^{k}f_{0}(\bx_p),
\end{eqnarray*}
which says that  $\check{R} [f] $ is the generalized partial-slice function whose restriction to  $\mathbb{R}^{p+1}$ is given
by $f_{0}(\bx_p)$. This result, combined with Lemma \ref{slice-monogenic-relation}, leads to a conclusion $\check{R} [f] \in \mathcal{AM}(\Omega)$.  Consequently, the  uniqueness  of generalized CK-extension, see Theorem \ref{generalized-CK-extension},
leads  to that $GCK[f_{0}]=\check{R} [f]=\check{R} \circ CK[f_{0}]$, as desired.
\end{proof}





\vskip 10mm
\end{document}